\documentclass{article}
\usepackage{graphicx}
\usepackage{amssymb}
\usepackage{amsmath}
\newcommand{\eq}{\begin{equation}}
\newcommand{\en}{\end{equation}}

\newcommand{\ed}{ \stackrel{d}{=}}

\newcommand{\mn}{\mathbb{N}}

\newtheorem{theorem}{\large Theorem} 
\newtheorem{proposition}[theorem]   {\large Proposition}

\newtheorem{lemma}[theorem]{\large  Lemma}

\font\tenmath=msbm10
\font\sevenmath=msbm7
\font\fivemath=msbm5
\newfam\mathfam \textfont\mathfam=\tenmath
\scriptfont\mathfam=\sevenmath \scriptscriptfont\mathfam=\fivemath

\def \\ { \cr }

\begin{document}

\title{Exponential--Uniform Identities Related to Records }
\author{Alexander Gnedin\thanks{Queen Mary University of London, email: a.gnedin@qmul.ac.uk}~~and~~Alexander Marynych\thanks{Eindhoven University of Technology, email: O.Marynych@tue.nl}}

\date{}
\maketitle

\begin{abstract}
\noindent 
We consider a rectangular grid induced by the south-west records
 from the planar Poisson point process  in ${\mathbb R}_+^2$.
A random symmetry property of the matrix
whose entries are the areas of  tiles of the grid implies
 cute multivariate distributional identities for
certain rational functions of independent
exponential and uniform random variables.
\end{abstract}


\paragraph{1. Introduction}
Let $E_1,E_2,\dots$ and $U_1,U_2,\dots$ be two independent sequences of independent rate-one exponential
and $[0,1]$-uniform random variables,
respectively.
A prototype of the distributional identities appearing in this note is the  identity
\begin{equation}\label{B=C}
\left({E_1\over U_1}+{E_2\over U_1U_2}+ \ldots+{E_n\over U_1\cdots U_n}\right) \left(1-U_1\cdots U_{n+1}\right)
\stackrel{d}{=}\left(E_1+{E_2\over U_1}+ \ldots+{E_{n+1}\over U_1\cdots U_n}\right) \left(1-U_1\cdots U_{n}\right),
\end{equation}
which  was used in \cite{PPP} to explain coincidence of 
the values in two quite different problems of optimal stopping.
Some probabilities related to 
the simplest instance of (\ref{B=C}),
\begin{equation}\label{B=Csmall}
{E_1\over U_1}(1-U_1U_2)\stackrel{d}{=}
\left(E_1+{E_2\over U_1}\right)(1-U_1),
\end{equation}
had been evaluated in \cite{Samuels}.

 We will show that (\ref{B=C}) along with more general identities for matrix functions in the exponential and uniform variables follow from  symmetry properties of the set of {\it records} (also known as Pareto-extremal points \cite{Yul}) from the planar Poisson process in the positive quadrant.
This continues the line of  \cite{Corners}, where it was argued that the planar Poisson process is a natural 
framework  for two  gems of combinatorial probability: Ignatov's theorem \cite{Ignatov}
and the Arratia-Barbour-Tavar{\'e} lemma on  the  scale-invariant Poisson processes on ${\mathbb R}_+$ \cite{ABTLemma}.


We shall evaluate the areas of tiles
for a rectangular grid induced by the set of records.
The identities obtained in this way are genuinely multivariate,
albeit  they stem from the arrays with identical marginal  distributions.
In particular, (\ref{B=Csmall}) appears by a row summation in
\begin{eqnarray}\label{cc}
\left( \begin{array}{cc}
(1-U_1)E_1&(1-U_1){E_2\over U_1}\\
\\
U_1(1-U_2)E_1& (1-U_2)E_2
\end{array}\right)\stackrel{d}{=}
\left( \begin{array}{cc}
(1-U_2)E_2&(1-U_1){E_2\over U_1}\\
\\
U_1(1-U_2)E_1& (1-U_1)E_1
\end{array}\right).
\end{eqnarray}

\paragraph{2. A random tiling induced by records}
 We use the self-explaining notations $\nearrow,\,\searrow,\,\swarrow,\,\nwarrow$ for four  
 coordinate-wise partial orders on the positive quadrant.
For instance, relations $a\nearrow b$ and  $b\swarrow a$ for $a,b\in{\mathbb R}^2_+$ both mean that $b$ is located strictly
north-east of $a$.

Let ${\cal P}$ be the planar Poisson point process with unit intensity in ${\mathbb R}_+^2$.
It will be convenient to understand $\cal P$ as a random set, rather than counting measure.
The event $(t,x)\in{\cal P}$ is interpreted as the value $x$ observed at time $t$.
Note that with probability one no two atoms of $\cal P$ lie on the same vertical or horizontal line.
An atom $r\in {\cal P}$ is said to be a (lower) {\it record} if  there is no earlier
observation with a smaller value, that is $a\nearrow r$ for no $a\in {\cal P}$.
The set of records, denoted henceforth $\cal R$,  is a point process with the intensity function $e^{-tx}$. 
The collection of records ordered by increase of the time component is a two-sided infinite $\searrow$-chain
$$\ldots\searrow r_{-2}\searrow r_{-1}\searrow r_{0} \searrow r_1\searrow r_2\searrow\ldots,$$
which we label by nonzero integers in such a way that the records  $r_{0}$ and $r_1$ are separated by the bisectrix $t=x$.

\par  Drawing vertical  lines at $t$-locations of records (record times), and drawing horizontal lines at their $x$-locations  (record values)
divides the positive quadrant into rectangular tiles.
Let $C_{ij}$ be the area of the tile whose north-west corner is
the intersection point
of the horizontal line through $r_i$ and the vertical line  through $r_j$ (see Figure 1). 
In particular, $C_{ii}$ for $i\in\mathbb{Z}$ is the area of a tile
spanned on  records $r_i, r_{i+1}$.

\begin{figure}
\begin{center}
\includegraphics[scale=0.6]{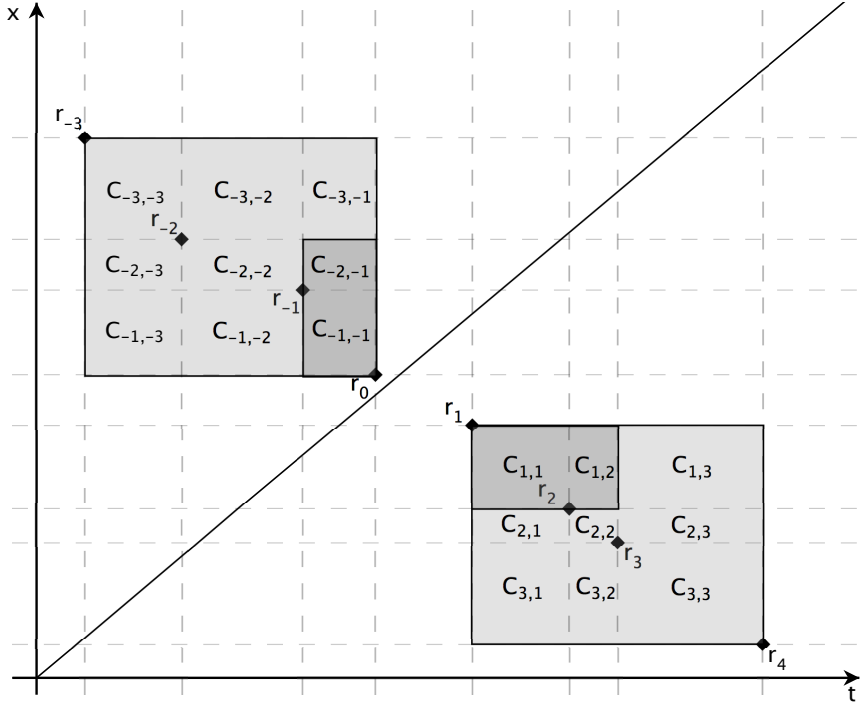}
\caption{The rectangular tiling and areas with the same distribution.}
\end{center}
\end{figure}

Given a record at location $(t,x)$, the next record is 
just the next point of $\cal P$ south-west of $(t,x)$, hence
 distributed like $(t+E/x, xU)$, as is easily seen from 
the independence and homogeneity properties of $\cal P$. 
The sequence $r_1, r_2,\dots$ is Markovian with  just described transitions, 
whence 
\begin{eqnarray}\label{R}
(C_{ij}\,;\, i,j=1,2,\ldots)&\stackrel{d}{=}&
\left(U_1\cdots U_{i-1}(1-U_i)\,{E_j\over U_1\cdots U_{j-1}}\,\,;~ i,j=1,2,\ldots\right).
\end{eqnarray}
Note that the left-hand side of (\ref{R}) is independent of $r_1$.
Since the law of $\cal R$ is 
not changed by reflection about the bisectrix, we also have
\begin{equation}
\label{Rrefl}
(C_{i,j}\,;\, i,j=1,2,\ldots) \stackrel{d}{=}
(C_{-j-1,-i-1}\,;\, i,j=1,2,\ldots).
\end{equation}
As an illustration, the areas of two rectangles (Figure 1) spanned on $r_1, r_4$ and $r_{-3},r_{0}$, respectively,  have the 
same distributions.

For $n\in\mn$ and $k\in\mathbb{Z}$ let 
$M_{k,n}
=(C_{k+i-1,k+j-1},\;\, i,j=1,\ldots,n)$ be the $n\times n$ matrix associated with records $r_{k},r_{k+1},\ldots,r_{k+n}$.
Obviously from the above, $M_{k,n}$ is independent of $r_k$ and satisfies $M_{k,n}\stackrel{d}{=}M_{1,n}$ 
for $k=1,2,\dots$.
For $k\leq 0$, $M_{k,n}$ is not independent of $r_k$, since 
the transition probability from $r_k=(t,x)$ to $r_{k+1}$ 
accounts for the condition that $-k$ records must lie south-east of $(t,x)$  
above the bisectrix.
Also, $M_{k,n}\stackrel{d}{=}M_{1,n}$ fails for $-n< k\leq 0$:
for instance, $C_{0,0}$ and $C_{1,1}$ have different distributions.
Nevertheless, we will show that $M_{k,n}\stackrel{d}{=}M_{1,n}$ holds for  $k\leq -n$, 
which by virtue of (\ref{Rrefl}) is equivalent to the following random symmetry property of $M_{1,n}$.

Let $M^\star_{1,n}$ be the matrix  
obtained by reflecting $M_{1,n}$ about the antidiagonal, 
that  is by exchanging each entry $(i,j)$ with entry $(n-j+1, n-i+1)$.

\begin{proposition}\label{P1}  For $n=1,2,\dots$
\begin{equation}\label{star}
M_{1,n}^{\star}\stackrel{d}{=} M_{1,n}.
\end{equation}
\end{proposition}

\vskip0.3cm
\noindent
Identity (\ref{cc}) is the $n=2$ instance of (\ref{star}). Identity (\ref{B=C}) appears by calculating the sum 
 of all entries of $M_{1,n+1}$ except the entries in the $(n+1)$st row.
Further identities can be derived by applying functions, e.g., taking the product of matrix elements in  the first row of $M_{1,n}$:
$${E_1\cdots E_n (1-U_1)^n\over U_1^{n-1}U_2^{n-2}\cdots U_{n-1}}
\stackrel{d}{=} {E_n^n(1-U_1)\cdots (1-U_n)\over U_1 U_2^2\cdots U_{n-1}^{n-1}}\,.$$


\paragraph{3. Records in a finite box}
To prove (\ref{star}) we consider records in finite rectangles.
Let $A\subset {\mathbb R}_+^2$  be a finite open rectangle with sides parallel to the coordinate axes.
Atom $a\in {\cal P}\cap A$ will be called  $A$-record if no other atom $b\in {\cal P}\cap A$ lies south-west of $a$. 
The set of $A$-records induces a random partition of $A$ in rectangular tiles. Denote by
$N=(N_{i,j})$ the random matrix of areas of the tiles.
The number of rows (or columns) of $N$ is a random variable that is one plus the number of $A$-records. 
Let  $N^{\star}=(N^*_{i,j})$ be the array obtained by reflecting $N$ about the  antidiagonal, 
which is defined conditionally on the number of $A$-records.

\begin{lemma}\label{L2} For every rectangle $A$ we have
\begin{itemize}
\item[\rm (i)]   $N\ed N^{\star}$,
\item[\rm(ii)]  also $N\ed N^{\star}$ conditionally given the number of $A$-records is $n$, for each $n\geq 0$,
\item[\rm(iii)] the distribution of $N$ depends on $A$ only through the area of $A$.
\end{itemize}
\end{lemma}
\begin{proof} 
Applying a hyperbolic shift $(t,x)\mapsto (\lambda t, x/\lambda)$ with some $\lambda>0$,
 rectangle $A$ can be mapped onto a square. The mapping preserves the area and the coordinate-wise orders, hence preserves the distribution of
$N$. For $A$ a square (i) and (ii) follow by symmetry of $A$ about the north-east diagonal (see Figure 2).

\end{proof}
\begin{figure}
\begin{center}
\includegraphics[scale=0.4]{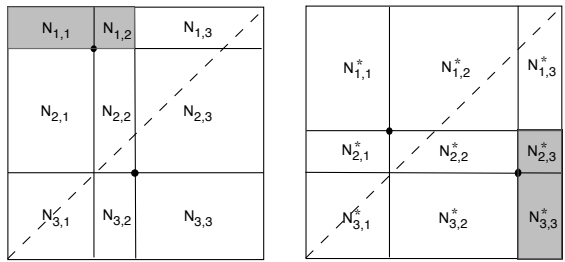}
\caption{Records in a square}
\end{center}
\end{figure}

\begin{lemma}\label{L3} For $n\geq 1$ the following conditional distributions coincide:
\begin{itemize}
\item[\rm(a)] the distribution of $M_{k,n}$ given the area $v$ of the rectangle spanned on records $r_k$ and $r_{n+k}$,  
where $k\geq 1$ or $k\leq -n-1$, 
\item[\rm(b)] the distribution of $N$ for a rectangle $A$ of area $v$, given that the number of $A$-records is $n-1$. 
\end{itemize}
\end{lemma}
\begin{proof} For any fixed rectangle $A$, the set of $A$-records is independent of the Poisson point process outside $A$.
On the other hand, given that the north-west and the south-east corners of $A$ are records,
$\cal P$ has no atoms south-east of these corners, hence the set of $A$-records coincides with ${\cal R}\cap A$. 
That is to say, given that two records are located at the corners of $A$, the records inside $A$ are distributed like $A$-records.
In the same way, taking,  
for instance, $k=1$, we have:  given that  $r_1$ and $r_{n+1}$ are at the corners of  rectangle $A$ (below the line $x=t$), 
the set $\{r_2,\dots,r_n\}$ has the same distribution as the set of $A$-atoms, conditioned on the event that the number of $A$-records is $n-1$. 
Now the assertion follows from Lemma \ref{L2} (iii).
\end{proof}

\vskip0.3cm
Proving Proposition \ref{P1} is now easy.
Combining Lemma \ref{L3} with Lemma \ref{L2} (ii) we see that the distributional identity (\ref{star}) holds conditionally on the area of
the rectangle spanned on $r_1$ and $r_{n+1}$, hence (\ref{star}) also holds unconditionally.
Note that the area is equal to the sum of all entries of  the $n\times n$ matrix, that is distributed like 

$$(1-U_1\cdots U_n)\left( E_1+\frac{E_2}{U_1}+  \frac{E_3}{U_1U_2}+\cdots+ \frac{E_n}{U_1\cdots U_{n-1}}  \right).$$

\par We could not find a proof of  (\ref{B=C}) by 
computing  densities or transforms, or by connecting to other known identities
 like `beta-gamma algebra' \cite{Dufresne}. Spanning a grid on the point process 
of $k$-corners \cite{Corners} we were able to show that similar  identities 
 hold with uniform distribution replaced by beta$(1,\theta)$.

\end {document}